\pdfoutput=1
\documentclass[journal,twoside,web]{ieeecolor}

\usepackage{generic}
\usepackage{lcsys}

\usepackage{graphicx}      
\usepackage[caption=false]{subfig}
\usepackage{cite}
\usepackage{mondeq}

\usepackage{amsthm}
\usepackage{cleveref}
\usepackage{url}

\newtheorem{definition}{Definition}
\newtheorem{theorem}{Theorem}
\newtheorem{proposition}{Proposition}
\newtheorem{corollary}{Corollary}

\begin{document}

\title{Quantization Robustness of Monotone Operator Equilibrium Networks}

\author{James Li, Philip H.W. Leong, and Thomas Chaffey
\thanks{The authors are with the School of Electrical and Computer Engineering, The University of Sydney, NSW, Australia.
Emails: \texttt{jali4795@uni.sydney.edu.au}, \texttt{\{philip.leong, thomas.chaffey\}@sydney.edu.au}.}
\thanks{\copyright\ 2026 IEEE. Personal use of this material is permitted. Permission from IEEE must be obtained for all other uses, in any current or future media, including reprinting/republishing this material for advertising or promotional purposes, creating new collective works, for resale or redistribution to servers or lists, or reuse of any copyrighted component of this work in other works. This is the authors' accepted version; the version of record will appear in IEEE Control Systems Letters via its DOI.}}

\maketitle
\thispagestyle{empty}

\begin{abstract}

Monotone operator equilibrium networks (MonDEQs) are implicit models whose well-posedness and convergence rely on a single scalar margin computed from the weight matrix, which has made them attractive for learned controllers with closed-loop stability guarantees. When such models are deployed on low-precision hardware, weights are quantized, and these structural guarantees can be destroyed. We provide an explicit quantization-preservation certificate: a single computable check on the weight perturbation determines whether the quantized network retains existence, uniqueness, linear convergence, and a bounded displacement of its equilibrium. The same threshold certifies the backward solve as well as the forward solve, so it governs both training and inference. MNIST experiments confirm a phase transition at the predicted threshold: three- and four-bit post-training quantization diverge, while five-bit and above converge; quantization-aware training recovers provable convergence at four bits.

\end{abstract}



\section{Introduction}

\IEEEPARstart{D}{eploying} neural networks as controllers in safety-critical applications requires rigorous behavioral guarantees. Implicit-layer architectures such as monotone operator equilibrium networks (MonDEQs)~\cite{winston_monotone_2021} and the related recurrent equilibrium networks (RENs)~\cite{revay_recurrent_2024} have emerged as models providing such guarantees: RENs have been used to learn nonlinear controllers with closed-loop stability~\cite{junnarkar_synthesis_2022, wang_learning_2022}. However, embedded deployment requires \textbf{quantization}: representing weights and activations at low-bit precision~\cite{nagel_white_2021}. Hardware efficiency and accuracy are conflicting goals, since rounding error grows as bit-precision is reduced. Analytic bounds relating quantization error to a network's robustness would let bit-width be selected based on deployment requirements rather than by trial and error.

This motivates the question of whether quantization error can be bounded at the model level. At present, there is no generally applicable bound on quantization error; instead, only architecture-specific analyses exist~\cite{zhang_qebverif_2023, kabaha_quantization_2025}. Progress therefore requires restricting attention to architectures with tractable convergence guarantees --- a requirement familiar in control, where quantized feedback has been modeled as a sector-bounded perturbation and stability is analyzed via small-gain conditions~\cite{fu_sector_2005}. MonDEQs are a class of deep equilibrium models (DEQs)~\cite{bai_deep_2019} that enforce monotonicity of the underlying operator, guaranteeing existence, uniqueness, and linear convergence of the equilibrium via operator splitting. A MonDEQ layer's well-posedness is captured by a single spectral margin: the smallest eigenvalue $m$ of a symmetric matrix constructed from the layer's weights (defined formally in Section~II). Having $m>0$ ensures the implicit equation has a unique equilibrium that the numerical solver converges to; because quantization perturbs this matrix, the margin $m$ provides a natural handle for analyzing quantization error. To our knowledge, MonDEQ behavior under quantization has not been analyzed.

\subsection{Contributions}

We give the first explicit quantization-preservation certificate for a MonDEQ's structural guarantees: a single computable threshold $\|\Delta W\|_2 < m$ under which existence, uniqueness, linear convergence, and a deterministic displacement bound all hold for the quantized network. The induced perturbation of the monotonicity margin and Lipschitz constant is bounded by specialising the radius theorem for monotone mappings~\cite[Thm.~4]{dontchevRadiusTheoremsMonotone2019} to the MonDEQ setting (Theorem~\ref{thm:weight_quant_bound}, Section~\ref{subsec:margin-wellposed}); this gives explicit conditions under which the quantized MonDEQ retains existence, uniqueness, and linear convergence (Corollary~\ref{cor:quantized_fb_contraction}, instantiating the forward--backward splitting framework of~\cite{bauschke_convex_2017}). The fixed-point displacement between quantized and full-precision equilibria is bounded and converted into a condition number (Theorems~\ref{thm:equilibrium_displacement}--\ref{thm:condition_number}, Section~\ref{subsec:displacement}), sharpening the closest prior MonDEQ Lipschitz analysis~\cite{pabbaraju_estimating_2021} by deriving the perturbed margin from the quantizer bit-width. The same threshold certifies the backward solve as well as the forward solve (Theorem~\ref{thm:backward_quantized}, Section~\ref{subsec:backward}), so a single margin check governs both training and inference. We validate the certificate empirically on a single-layer MonDEQ trained on MNIST across bit-widths from 3 to 32 bits (Section~\ref{sec:experiments}); the experiments test certificate predictiveness against the threshold $\|\Delta W\|_2 < m$, not deployment-scale benchmark performance. Code is available at \url{https://github.com/JLi-Projects/mondeq-quant}.

\subsection{Related Work}

\emph{Quantization theory.}
Standard quantization modeling treats the quantized weight matrix as a bounded perturbation of its full-precision counterpart~\cite{higham_accuracy_2002,nagel_white_2021}; post-training and quantization-aware variants (Section~\ref{sec:experiments}) trade off training cost against achievable bit-width~\cite{jacob_quantization_2017, nagel_white_2021}.
\emph{Inexact operator splitting.}
Operator splitting methods such as forward--backward and Peaceman--Rachford admit inexact variants in which bounded per-step errors are tolerated while preserving convergence~\cite{eckstein_douglasrachford_1992,combettes_proximal_2011}. In Section~\ref{sec:quant-mondeq}, we apply these results to quantization-induced errors in the MonDEQ solver and derive new bounds on equilibrium displacement and the associated condition number.

\emph{Numerical error analysis.}
Beuzeville et al.~\cite{beuzeville_deterministic_2025} prove backward stability of feedforward networks under floating-point rounding; Jonkman et al.~\cite{jonkman_quantisation_2018} model quantized communication in distributed optimization as inexact Krasnosel'ski\u\i--Mann iteration.

\emph{MonDEQ sensitivity.}
Pabbaraju et al.~\cite{pabbaraju_estimating_2021} derive input-output and weight-output Lipschitz bounds for MonDEQs, but their perturbation bound assumes the perturbed margin is known and does not address quantization-specific structure, convergence conditions, or condition number.

\section{Preliminaries}

We collect notation and standard definitions from monotone operator theory that are used throughout the paper.

We work in $\R^n$ with the Euclidean norm $\|\cdot\|_2$ and denote the spectral norm of a matrix by $\|\cdot\|_2$. The symmetric and skew-symmetric components of a matrix $A$ are $\sym(A) := \tfrac{1}{2}(A + A^\top)$ and $\skw(A) := \tfrac{1}{2}(A - A^\top)$.

\emph{Monotone operators.}
Monotonicity generalises positive-definiteness from linear maps to nonlinear ones and is the property that guarantees well-posedness of the fixed-point equations we analyse.
An operator $F:\R^n \to \R^n$ is \emph{monotone} if $\ip{F(x) - F(y)}{x - y} \ge 0$ for all $x,y \in \R^n$, \emph{maximal} if its graph is not properly contained in the graph of any other monotone operator, \emph{$m$-strongly monotone} if $\ip{F(x) - F(y)}{x - y} \ge m\|x - y\|_2^2$, and \emph{$L$-Lipschitz} if $\|F(x) - F(y)\|_2 \le L\|x - y\|_2$. For the affine operator $F(z) = (I - W)z - (Ux + b)$, the strong monotonicity margin is $m = \lmin(\sym(I - W))$ and the Lipschitz constant is $L = \|I - W\|_2$~\cite{ryu_primer_2016,bauschke_convex_2017}.

\emph{Resolvents.}
The resolvent $\res{\alpha G} := (I + \alpha G)^{-1}$ of a maximal monotone $G$ plays the role of the activation function inside the splitting iteration; it is single-valued, firmly nonexpansive, and hence $1$-Lipschitz. The reflected resolvent is $R_{\alpha G} := 2\res{\alpha G} - I$~\cite{bauschke_convex_2017}.

The \emph{forward--backward} iteration $z^{k+1} = \res{\alpha G}\bigl(z^k - \alpha F(z^k)\bigr)$ converges linearly for any $\alpha \in (0, 2m/L^2)$ with contraction modulus $r_{\mathrm{FB}} = \sqrt{1 - 2\alpha m + \alpha^2 L^2}$~\cite{ryu_primer_2016,bauschke_convex_2017}. The \emph{Peaceman--Rachford} iteration $z^{k+1} = (2\res{\alpha G} - I)\bigl((2\res{\alpha F} - I)(z^k)\bigr)$ converges linearly for any $\alpha > 0$ with contraction modulus $\rho_{\mathrm{PR}} = \sqrt{1 - \tfrac{4\alpha m}{(1+\alpha L)^2}}$~\cite{ryu_primer_2016,bauschke_convex_2017}.

\section{Monotone Operator Equilibrium Networks}

Monotone operator equilibrium networks (MonDEQs)~\cite{winston_monotone_2021} compute their output as the fixed point of a splitting map derived from a monotone inclusion. We summarize the key definitions.

\begin{definition}\label{def:mondeq}
Fix an input $x \in \R^d$. Let $W\in\R^{n\times n}$, $U\in\R^{n\times d}$ and $b\in\R^{n}$ be parameters collected in a vector $\boldsymbol{\vartheta} \in \R^r$. Define the affine map
\begin{equation*}
    F(z)\;:=\;(I-W)z-(Ux+b),\qquad z\in\R^{n}.
\end{equation*}
Let $G : \R^n \rightrightarrows \R^n$ be a maximal monotone operator and let $\res{\alpha G}:=(I+\alpha G)^{-1}$ denote its resolvent for any $\alpha>0$. Considering the nonlinear fixed point iteration
\begin{equation*}
    z^{k+1} = \res{\alpha G}\bigl(z^k - \alpha F(z^k)\bigr) := \Phi(z^k; \boldsymbol{\vartheta}),
\end{equation*}
suppose it has a fixed point $\zstar$. We call the mapping from the input $x$ to fixed point $\zstar$ a \textbf{monotone operator equilibrium network} (MonDEQ).
\end{definition}

A MonDEQ thus replaces the layer stack of a feedforward network with a single nonlinear fixed-point equation: the iteration $\Phi$ alternates an affine pre-activation step (parameterised by $W$, $U$, $b$) with the resolvent of $G$ (which plays the role of the activation), and the network's output is its equilibrium. The following equivalence~\cite{winston_monotone_2021} recasts this fixed point as the solution of a monotone inclusion, opening the door to operator-splitting theory.

\begin{theorem}\label{thm:monotone_inclusion}
Define a MonDEQ as in Definition~\ref{def:mondeq}. Then $\zstar \in \Fix(\Phi) \Longleftrightarrow 0 \in F(\zstar) + G(\zstar)$.
\end{theorem}

Theorem~\ref{thm:monotone_inclusion} reduces computation of the MonDEQ output to solving the monotone inclusion $0 \in F(\zstar) + G(\zstar)$. This reformulation is useful because the splitting algorithms of monotone operator theory apply directly and converge linearly when $F$ is strongly monotone. The choice of $G$ encodes the activation: when $G = \partial\rho$ for proper, closed, convex $\rho$, the resolvent $\res{\alpha G} = \operatorname{prox}_{\alpha\rho}$ acts as the activation function in the splitting iteration, with $\rho$ the indicator of $\R_{\ge 0}^n$ recovering the rectified linear unit (ReLU) activation~\cite{winston_monotone_2021}.

Any $W$ guaranteeing $F$ is $m$-strongly monotone can be written in the form below, giving a constructive recipe for training MonDEQs with a target margin built in.

\begin{proposition}\label{prop:weight_mtrx}
    $\sym(I - W) \succeq mI$ if and only if there exist $A,B \in \R^{n \times n}$ such that $W \;=\; (1 - m)I - A^\top A + B - B^\top$.
\end{proposition}

\begin{proof}
Direct computation~\cite{winston_monotone_2021}.
\end{proof}

The margin $m$ is determined by the parameterization of $W$. Because $m = \lmin(\sym(I - W))$ is an explicit function of $W$, perturbing the weight matrix perturbs $m$ in a way that can be bounded analytically. Since $m > 0$ is both necessary and sufficient for well-posedness, bounding how quantization perturbs $m$ directly determines whether the quantized network remains well-posed.

\section{Quantization in a MonDEQ}\label{sec:quant-mondeq}

Here, quantization replaces floating-point weights with fixed-point (low-bit) approximations, reducing memory and enabling efficient integer arithmetic at the cost of increased rounding error. We analyze the resulting error as a perturbation of the weight matrix $W \to \Wq = W + \DW$~\cite{higham_accuracy_2002}, bounding its effect on well-posedness, the equilibrium point, and the backward pass used for training.

We use symmetric uniform (mid-tread) quantization: for $b$-bit representation with weights in $[-1,1]$, the quantizer $Q_\Delta(w) = \Delta \cdot \mathrm{round}(w/\Delta)$ has step size $\Delta = 2^{1-b}$ and worst-case elementwise error $\Delta/2$. Uniform quantization is standard for weight compression because the evenly spaced levels map directly to fixed-point integer formats, enabling hardware-accelerated matrix arithmetic; non-uniform schemes such as logarithmic quantizers~\cite{fu_sector_2005} sacrifice this property. Since each entry of $\Delta W$ is bounded by $\Delta/2$, we have $\norm{\DW}_2 \le (\Delta/2)n$ for the square $n \times n$ weight matrix. This motivates modeling weight quantization as a bounded perturbation~\cite{higham_accuracy_2002}.

\begin{definition}\label{def:weight_quant}
    Given a MonDEQ as in Definition~\ref{def:mondeq}, its \emph{quantized counterpart} replaces $W$ with $\Wq = W + \DW$, $\|\DW\|_2 \le \varepsilon_W$.
\end{definition}
For the symmetric uniform quantizer with step size $\Delta = 2^{1-b}$ at $b$~bits, $\varepsilon_W = n\Delta/2$.

Weight quantization introduces a deterministic perturbation to the weight matrix. This raises the question of how large the perturbation can be before the equilibrium ceases to exist. In practice, each iterate also incurs computational errors such as finite-precision arithmetic or activation rounding, so the computed iterates obey $z^{k+1} = \widetilde{\Phi}(z^k) + \delta_k$ with bounded per-step errors $\delta_k$. Together, the weight perturbation $\DW$ and the iterate errors $\delta_k$ model the two sources of error in a quantized MonDEQ.

\subsection{Margin Perturbation and Well-Posedness}\label{subsec:margin-wellposed}

The following theorem shows that weight perturbation reduces the monotonicity margin by at most $\norm{\DW}_2$. This theorem is a specialisation of the radius theorem for monotone mappings~\cite[Theorem~4]{dontchevRadiusTheoremsMonotone2019}.

\begin{theorem}\label{thm:weight_quant_bound}
    Define a MonDEQ in accordance with Definition~\ref{def:mondeq} with weights $W$ satisfying Proposition~\ref{prop:weight_mtrx}. Let $\Wq$ be the quantized weights with perturbation $\|\DW\|_2\le\varepsilon_W$, and let $\Fq(z) := (I - \Wq)z - (Ux + b)$ denote the corresponding quantized affine operator. Then the margin $\mq := \lmin(\sym(I - \Wq))$ of $\Fq$ is bounded below by
    \begin{equation*}
        \mq \;\ge\; m - \|\DW\|_2,
    \end{equation*}
    and the Lipschitz constant $\Lq$ of $\Fq$ satisfies $|L-\|\DW\|_2| \le \Lq \le L+\|\DW\|_2$. In particular, $\Fq$ is strongly monotone (with margin $\mq > 0$) whenever $\|\DW\|_2 < m$.
\end{theorem}

\begin{proof}
$\sym(I - \Wq) = \sym(I-W) - \sym(\DW)$, so by Rayleigh~\cite{Horn_Johnson_1985},
\begin{align*}
  \mq
    &= \min_{\norm{x}_2=1} x^\top\bigl[\sym(I-W) - \sym(\DW)\bigr]x \\
    &\ge m - \norm{\sym(\DW)}_2 \ge m - \norm{\DW}_2.
\end{align*}
For the Lipschitz constant, the triangle and reverse triangle inequalities applied to $\Lq = \norm{I-\Wq}_2 = \norm{(I-W) - \DW}_2$ give the stated bounds.
\end{proof}

If $\norm{\DW}_2 < m$ then $\mq > 0$ and the equilibrium is preserved; in the worst case the condition number $\widetilde\kappa = \Lq/\mq$ degrades from both sides, slowing convergence. The margin is the binding constraint in practice: $\sym(I-W) = mI + A^\top A$ attains $m$ exactly wherever $A^\top A$ has a zero eigenvalue, while $L = \norm{I-W}_2$ is robust to elementwise rounding. The Peaceman--Rachford analogue substitutes $\rho_{\mathrm{PR}}(\alpha;\mq,\Lq)$.

\begin{corollary}\label{cor:quantized_fb_contraction}
If $\varepsilon_W < m$ and $\alpha \in (0, 2\mq/\Lq^2)$, the quantized forward--backward map $\widetilde\Phi_{\mathrm{FB}}(z) := \res{\alpha G}\bigl(z - \alpha \Fq(z)\bigr)$ is a contraction with modulus $r_{\mathrm{FB}}(\alpha;\mq,\Lq)$.
\end{corollary}

\begin{proof}
    Replace $(m, L)$ by $(\mq, \Lq)$ from Theorem~\ref{thm:weight_quant_bound} in the forward--backward convergence rate.
\end{proof}

In words, provided $\varepsilon_W < m$, weight quantization slows but does not break convergence: the solver still reaches a unique equilibrium, and the next subsection bounds how far that equilibrium moves. Larger perturbations can drive $\mq \le 0$ and break convergence, as in the 4-bit case of Section~\ref{sec:experiments}.

\subsection{Equilibrium Displacement}\label{subsec:displacement}

Convergence guarantees the quantized solver reaches \emph{some} fixed point, but a controller deployed at low precision needs to know how far that fixed point has moved from the one the controller was designed for. The next result bounds the displacement $\norm{\zqstar - \zstar}_2$ in terms of the perturbation size and the (unperturbed) margin.

\begin{theorem}
\label{thm:equilibrium_displacement}
Assume $F(z) = (I-W)z - (Ux+b)$ is $m$-strongly monotone and $G:\R^n \rightrightarrows \R^n$ is monotone. With $\Wq$ given as in Definition~\ref{def:weight_quant}, suppose $\|\DW\|_2 < m$ (in particular $\mq > 0$). Let
\begin{equation*}
    \Fq(z) := (I-\Wq)z - (Ux+b) = F(z) - \DW z.
\end{equation*}
Let $\zstar$ and $\zqstar$ denote the (unique) solutions of the full-precision and quantized inclusions
\begin{equation*}
    0 \in F(\zstar) + G(\zstar),\qquad
    0 \in \Fq(\zqstar) + G(\zqstar).
\end{equation*}
Then
\begin{equation}
\label{eq:equilibrium_displacement_bound}
    \|\zqstar - \zstar\|_2
    \;\le\;
    \frac{\|\DW\|_2}{m}\,\|\zqstar\|_2.
\end{equation}
\end{theorem}

\begin{proof}
Pick $g^\star \in G(\zstar)$, $\widetilde g^\star \in G(\zqstar)$ with $F(\zstar)+g^\star = 0$, $\Fq(\zqstar)+\widetilde g^\star = 0$. Subtracting (using $\Fq = F - \DW$) and taking the inner product with $\delta z := \zqstar - \zstar$,
\begin{equation*}
    \ip{F(\zqstar) - F(\zstar)}{\delta z}
    - \ip{\DW \zqstar}{\delta z}
    + \ip{\widetilde g^\star - g^\star}{\delta z} = 0.
\end{equation*}
The first term is $\ge m\norm{\delta z}_2^2$ ($m$-strong monotonicity of $F$); the third is $\ge 0$ (monotonicity of $G$). Hence $m\norm{\delta z}_2^2 \le \norm{\DW}_2\norm{\zqstar}_2\norm{\delta z}_2$ by Cauchy--Schwarz; dividing yields~\eqref{eq:equilibrium_displacement_bound}.
\end{proof}

The bound~\eqref{eq:equilibrium_displacement_bound} depends on $\norm{\zqstar}_2$ rather than $\norm{\zstar}_2$ because the perturbation acts through the shifted fixed point. Exchanging $F$ and $\Fq$ in the proof gives the symmetric bound $\norm{\zstar - \zqstar}_2 \le (\norm{\DW}_2/\mq)\,\norm{\zstar}_2$. An explicit \emph{relative} bound in terms of $\norm{\zstar}_2$ alone is given in Corollary~\ref{cor:relative_forward_error}.

For hardware deployment, $U$ and $b$ are also quantized; the same argument with $\Delta u := \Delta U\,x + \Delta b$ extends the bound.

\begin{corollary}\label{cor:full_parameter_displacement}
Under the hypotheses of Theorem~\ref{thm:equilibrium_displacement} with $\widetilde U = U + \Delta U$, $\widetilde b = b + \Delta b$, and $\Delta u := \Delta U\,x + \Delta b$,
$\norm{\zstar - \zqstar}_2 \le (\norm{\DW}_2\norm{\zqstar}_2 + \norm{\Delta u}_2)/m$.
\end{corollary}
\begin{proof}
The argument of Theorem~\ref{thm:equilibrium_displacement} applies with $\Fq(z) - F(z) = -\DW z - \Delta u$; Cauchy--Schwarz on the additional $\ip{\Delta u}{\delta z}$ term yields the extra $\norm{\Delta u}_2$.
\end{proof}
In words, quantizing $U$ and $b$ shifts the equilibrium by at most $\norm{\Delta u}_2/m$ but does not threaten convergence: the margin and Lipschitz constant of Theorem~\ref{thm:weight_quant_bound} depend only on $(I - \Wq)$, so Corollary~\ref{cor:quantized_fb_contraction}'s convergence guarantee extends verbatim. The next result accounts for the second error source, \emph{iterate quantization}: per-step residuals from finite-precision arithmetic or activation rounding.

\begin{corollary}
\label{cor:inexact_quantized_iteration}
Let $\widetilde{\Phi}$ be a quantized map as in Corollary~\ref{cor:quantized_fb_contraction}, with contraction modulus $r \in (0,1)$ and fixed point $\zqstar$. Then
\begin{equation*}
    \limsup_{k\to\infty}\|z^k - \zqstar\|_2
    \;\le\;
    \frac{\limsup_{k\to\infty}\|\delta_k\|_2}{1 - r}.
\end{equation*}
If $\sum_{k=0}^{\infty}\norm{\delta_k}_2 < \infty$, then $z^k \to \zqstar$ exactly.
\end{corollary}

\begin{proof}
Follows from standard inexact contraction results~\cite[Sec.~5.5]{bauschke_convex_2017}.
\end{proof}

In practice, bounded per-step errors do not destroy convergence: the solver reaches a neighbourhood of $\zqstar$ whose radius scales with the error magnitude and contraction rate. Summability $\sum\norm{\delta_k}_2 < \infty$ holds, for example, under an \emph{adaptive quantizer} (one whose step size shrinks across iterations) chosen so that $\norm{\delta_k}_2$ decays geometrically~\cite{jonkman_quantisation_2018}, in which case the total error $\norm{z^k - \zstar}_2 \le \norm{z^k - \zqstar}_2 + \norm{\zqstar - \zstar}_2$ collapses to the displacement bound alone.

The bound~\eqref{eq:equilibrium_displacement_bound} measures displacement in absolute terms. We now derive a relative bound and extract the condition number, which separates the problem's inherent sensitivity from the perturbation size.

\begin{corollary}\label{cor:relative_forward_error}
Under the hypotheses of Theorem~\ref{thm:equilibrium_displacement}, if $\norm{\DW}_2 < m$ then
\begin{equation}\label{eq:relative_forward_error}
    \frac{\norm{\zstar - \zqstar}_2}{\norm{\zstar}_2}
    \;\leq\;
    \frac{\norm{\DW}_2}{m - \norm{\DW}_2}.
\end{equation}
\end{corollary}

\begin{proof}
From Theorem~\ref{thm:equilibrium_displacement},
$\norm{\zqstar - \zstar}_2 \leq \frac{\norm{\DW}_2}{m}\norm{\zqstar}_2$.
Substituting $\norm{\zqstar}_2 \leq \norm{\zstar}_2 + \norm{\zqstar - \zstar}_2$ gives
$\norm{\zqstar - \zstar}_2 \leq \frac{\norm{\DW}_2}{m}(\norm{\zstar}_2 + \norm{\zqstar - \zstar}_2)$.
Rearranging, $(1 - \norm{\DW}_2/m)\norm{\zqstar - \zstar}_2 \leq \frac{\norm{\DW}_2}{m}\norm{\zstar}_2$, which yields~\eqref{eq:relative_forward_error} since $\norm{\DW}_2 < m$.
\end{proof}

Corollary~\ref{cor:relative_forward_error} gives a global bound: the relative displacement is at most $\norm{\DW}_2/(m - \norm{\DW}_2)$, which depends only on the perturbation size and margin. For example, on the trained MonDEQ of Section~\ref{sec:experiments} at 8~bits ($\norm{\DW}_2 = 0.035$, $m = 0.227$), the bound gives $18\%$; the empirical relative error is much smaller (Section~\ref{sec:experiments}). As $\norm{\DW}_2 \to 0$, the bound linearizes to $\norm{\DW}_2/m$, recovering the condition number scaling of Theorem~\ref{thm:condition_number}.

The sensitivity of the equilibrium to small weight perturbations is captured by the condition number~\cite{higham_accuracy_2002, beuzeville_backward_2024}.

\begin{theorem}\label{thm:condition_number}
For an unquantized MonDEQ with margin $m > 0$, the absolute condition number 
\begin{equation*}
    \kappaAbs := \limsup_{\norm{\DW}_2 \to 0} \frac{\norm{\zqstar - \zstar}_2}{\norm{\DW}_2}
\end{equation*}
satisfies $\kappaAbs \le \norm{\zstar}_2 / m$.
\end{theorem}

\begin{proof}
From Theorem~\ref{thm:equilibrium_displacement},
$\norm{\zstar - \zqstar}_2/\norm{\DW}_2 \leq \norm{\zqstar}_2/m$.
By Corollary~\ref{cor:relative_forward_error}, $\norm{\zqstar}_2 \leq \norm{\zstar}_2/(1 - \norm{\DW}_2/m)$, so
$\norm{\zstar - \zqstar}_2/\norm{\DW}_2 \leq \norm{\zstar}_2/(m - \norm{\DW}_2)$.
Taking $\norm{\DW}_2 \to 0$ gives $\kappaAbs \leq \norm{\zstar}_2/m$.
\end{proof}

In words, the equilibrium's sensitivity to weight perturbation is governed by the ratio of its magnitude to the margin. The relative condition number $\kappaRel \le \norm{W}_2/m$ gives $\norm{\zstar - \zqstar}_2/\norm{\zstar}_2 \leq \kappaRel\,\eta_W$ to first order, where $\eta_W := \norm{\DW}_2/\norm{W}_2$; on the trained MNIST model $\kappaRel \approx 7.6$. A sufficient pre-deployment check is $\varepsilon_W < m$, and a single margin check $\mq > 0$ guarantees both forward and backward convergence (Theorem~\ref{thm:backward_quantized}).

Unlike feedforward networks, where rounding errors accumulate through $L$ layers as $O(Lu)$~\cite{beuzeville_deterministic_2025}, contractivity bounds the error here regardless of iteration count: $\zqstar$ is exact for $I - \Wq$. The next subsection shows the backward solve inherits these guarantees verbatim.

\subsection{Backward Pass Under Quantization}\label{subsec:backward}

Training a MonDEQ requires gradients of the loss with respect to the parameters $\boldsymbol{\vartheta} = (W, U, b)$, computed by implicit differentiation through the equilibrium condition $0 \in F(\zstar;\boldsymbol{\vartheta}) + G(\zstar)$. The key observation is that the resulting backward problem is itself a monotone inclusion with the same linear part $(I - W)$ as the forward problem, and therefore inherits the same margin and convergence guarantees.

Differentiating $F(\zstar;\boldsymbol{\vartheta}) + g^\star = 0$ in $\boldsymbol{\vartheta}$ via the chain rule produces the backward inclusion $0 \in (I-W)p - r + G_b(p)$, where $p := \td{\zstar}{\boldsymbol{\vartheta}}$ is the backward sensitivity, $r := \td{W}{\boldsymbol{\vartheta}}\zstar + \td{U}{\boldsymbol{\vartheta}}x + \td{b}{\boldsymbol{\vartheta}}$ collects the parameter-derivatives of the affine forcing, and $G_b \in \partial_C G(\zstar) \subseteq \mathbb{R}^{n \times n}$ is an element of the Clarke generalized Jacobian --- the convex hull of limits of Jacobians at points of differentiability --- satisfying $\sym(G_b) \succeq 0$~\cite{combettes_deep_2019}. The following theorem shows this structure is preserved under weight quantization.

\begin{theorem}\label{thm:backward_quantized}
Let $\Wq = W + \DW$ with $\norm{\DW}_2 < m$, and let $\zqstar$ solve $0 \in \Fq(\zqstar;\boldsymbol{\vartheta}) + G(\zqstar)$. Define $\widetilde{p} := \td{\zqstar}{\boldsymbol{\vartheta}}$, $\widetilde{r} := \td{\Wq}{\boldsymbol{\vartheta}}\zqstar + \td{U}{\boldsymbol{\vartheta}}x + \td{b}{\boldsymbol{\vartheta}}$, and let $\widetilde{G}_b \in \partial_C G(\zqstar) \subseteq \mathbb{R}^{n \times n}$ with $\sym(\widetilde{G}_b) \succeq 0$. Then $\widetilde{p}$ solves
\begin{equation}
  0 \in (I-\Wq)\widetilde{p} - \widetilde{r} + \widetilde{G}_b(\widetilde{p}),
  \label{eq:backward_inclusion}
\end{equation}
and under the stepsize hypothesis of Corollary~\ref{cor:quantized_fb_contraction}, the splitting method converges to $\widetilde{p}$ with the perturbed parameters $(\mq,\Lq)$ from Theorem~\ref{thm:weight_quant_bound}. In particular, if the forward pass converges ($\mq > 0$), then the backward pass also converges with the same contraction modulus; a single margin check suffices for both passes.
\end{theorem}

\begin{proof}
Differentiating $\Fq(\zqstar;\boldsymbol{\vartheta}) + \widetilde{g}^\star = 0$ with respect to $\boldsymbol{\vartheta}$ yields~\eqref{eq:backward_inclusion}. The backward operator $(I - \Wq)p - \widetilde{r}$ has the same linear part as $\Fq$, so it inherits the same $(\mq, \Lq)$ from Theorem~\ref{thm:weight_quant_bound}. Since $\sym(\widetilde{G}_b) \succeq 0$ by hypothesis --- equivalently, $\widetilde{G}_b$ is monotone as a linear operator --- the same splitting method converges.
\end{proof}

Theorem~\ref{thm:backward_quantized} validates quantization-aware training (QAT): whenever the forward pass converges under quantized weights, gradients can be computed at the same precision and with the same iteration budget. No additional solver resources are required for the backward pass.

The gradient error under quantization has two sources: the displaced equilibrium ($\zstar \to \zqstar$) and the perturbed weight matrix ($W \to \Wq$). By Theorem~\ref{thm:backward_quantized}, the backward sensitivity $\widetilde{p}$ solves a monotone inclusion with the same linear operator $(I - \Wq)$, so the backward equilibrium exists and can be computed by the same splitting method. Since both sources introduce perturbations of size $O(\norm{\DW}_2)$ (the weight perturbation directly, and the equilibrium displacement via Theorem~\ref{thm:equilibrium_displacement}), the chain rule gives $\norm{\pd{\ell}{W} - \pd{\ell}{\Wq}}_2 = O(\norm{\DW}_2)$.

\section{Numerical Experiments}\label{sec:experiments}

We validate the predictions of Section~\ref{sec:quant-mondeq} on a single-layer MonDEQ with $n = 100$ hidden units trained on MNIST (Adam, $\mathrm{lr}=10^{-3}$, 15 epochs, step decay $\gamma=0.1$ at epoch~10). Unlike~\cite{winston_monotone_2021}, which fixes $m$, we treat $m$ as learnable via $m = \operatorname{softplus}(m_{\mathrm{raw}})$ with $m_{\mathrm{raw}} \in \R$, ensuring $m > 0$. The trained model achieves $98.22\%$ test accuracy with $m = 0.227$, $L = 1.845$, $\kappa = L/m = 8.13$. Post-training quantization (PTQ) applies symmetric uniform quantization with step $\Delta = 2^{1-b}$ and per-tensor scaling, without calibration or bias correction~\cite{jacob_quantization_2017, nagel_white_2021}; QAT retrains from scratch with a straight-through estimator. The forward--backward solver terminates when the relative residual falls below $10^{-5}$ or after 2000 iterations.

\begin{figure}[t]
\centering
\includegraphics[width=\columnwidth]{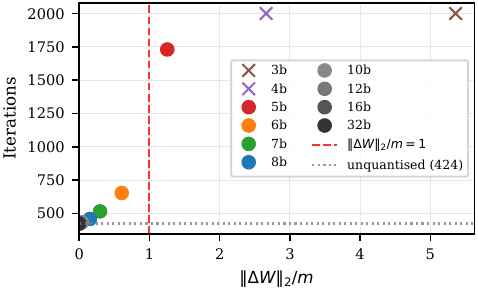}\\
\includegraphics[width=\columnwidth]{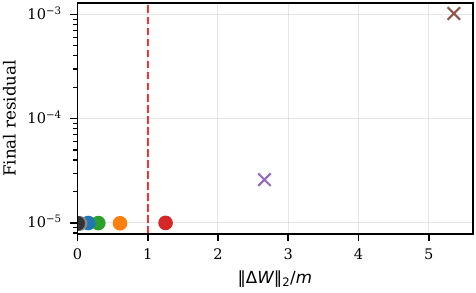}
\caption{Margin stability certificate. Iterations to convergence (top) and final residual (bottom) vs.\ normalized perturbation $\norm{\DW}_2/m$; each point is one bit-width (3--32~bits). The vertical dashed line marks the sufficient condition $\norm{\DW}_2/m = 1$; the horizontal dotted line in the top panel is the unquantized baseline (424 iterations). Circles: converged (relative residual ${<}\,10^{-5}$); crosses: did not converge within 2000 iterations.}
\label{fig:margin_stability}
\end{figure}

\textbf{Margin stability certificate.}
Figure~\ref{fig:margin_stability} tests the convergence condition $\norm{\DW}_2 < m$ from Theorem~\ref{thm:weight_quant_bound} across bit-widths 3--32. The non-convergence/convergence transition aligns with $\norm{\DW}_2/m = 1$: 3-bit (ratio $5.36$) and 4-bit ($2.66$) diverge, 5-bit and above converge. The 5-bit case (ratio $1.25$, $\mq = 0.045 > 0$) illustrates that the condition is sufficient but not necessary: the actual margin remains positive, so the solver converges despite the sufficient condition being violated. Iteration count reflects the degraded margin (5-bit ${\sim}1730$, 8-bit ${\sim}450$). At 8 bits, weight storage drops $4\times$ versus 32-bit floating-point with $98.24\%$ vs.\ $98.22\%$ accuracy.

\begin{figure}[t]
\centering
\includegraphics[width=\columnwidth]{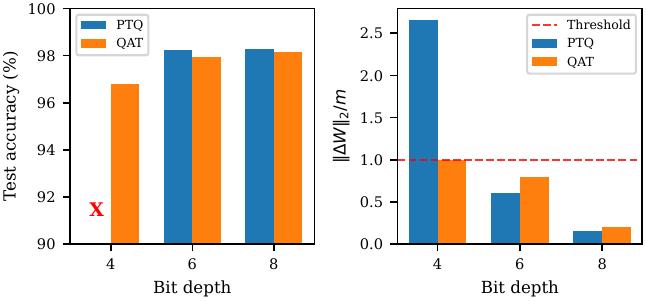}
\caption{QAT vs.\ PTQ at 4, 6, and 8~bits. Left: test accuracy (\%; a red~X indicates PTQ non-convergence at 4~bits). Right: $\norm{\DW}_2/m$; the dashed line marks $\norm{\DW}_2/m = 1$.}
\label{fig:qat_ptq}
\end{figure}

\begin{figure}[t]
\centering
\includegraphics[width=\columnwidth]{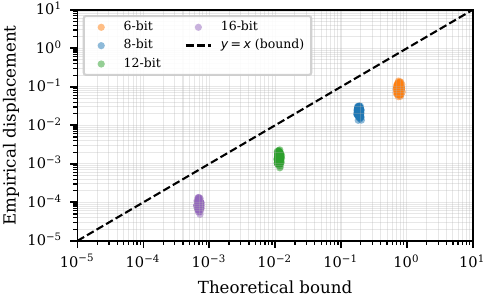}
\caption{Displacement bound validation at 6, 8, 12, and 16~bits with $W$, $U$, and $b$ all quantized. Each point is one test sample; axes show relative quantities (log--log). The $x$-axis is the Corollary~\ref{cor:full_parameter_displacement} bound $(\norm{\DW}_2\norm{\zqstar}_2 + \norm{\Delta u}_2)/(m\norm{\zstar}_2)$ with $\Delta u := \Delta U\,x + \Delta b$; the $y$-axis is the empirical relative displacement $\norm{\zqstar - \zstar}_2/\norm{\zstar}_2$. Points below the dashed line ($y = x$) satisfy the bound.}
\label{fig:splitting_displacement}
\end{figure}

\textbf{QAT vs.\ PTQ.}
Theorem~\ref{thm:backward_quantized}'s backward-pass guarantee makes QAT well-defined (it requires differentiating through the equilibrium). Figure~\ref{fig:qat_ptq} compares PTQ and QAT at 4, 6, 8~bits. PTQ fails at 4 bits ($\mq = -0.142$); QAT learns weights with $\mq = 0.006 > 0$, achieving $96.78\%$ accuracy at a smaller margin ($m = 0.184$ vs.\ $0.227$). At 6--8 bits both methods converge, with PTQ slightly higher ($98.25/98.29\%$) by inheriting the larger float margin.

\textbf{Displacement bound validation.}
The preceding experiments test \emph{convergence}; we now test the \emph{accuracy} of the converged equilibrium. Theorem~\ref{thm:equilibrium_displacement} bounds the displacement between exact equilibria; the forward--backward solver terminates at finite tolerance, so a Cauchy--Schwarz residual-to-state argument $\norm{\hat z - \zstar}_2 \le \norm{F(\hat z) + \hat g}_2/m$ for any $\hat g \in G(\hat z)$ (immediate from $m$-strong monotonicity, cf.~\cite[Sec.~5.5]{bauschke_convex_2017}) combined with Theorem~\ref{thm:equilibrium_displacement} gives a corrected observable bound that absorbs the solver tolerance. Figure~\ref{fig:splitting_displacement} illustrates Corollary~\ref{cor:full_parameter_displacement}'s bound on 2{,}560 randomly sampled test inputs at 6, 8, 12, and 16~bits, with $W$, $U$, $b$ all quantized at the same bit-width. The maximum $\norm{\Delta u}_2$ ranges from $2.44$ (6-bit) to $0.002$ (16-bit), and the empirical displacement is $3$--$10\times$ lower than the bound for every sample.

\section{Conclusions}

We have analyzed the effect of weight quantization on monotone operator equilibrium networks through spectral perturbation of the monotone inclusion. The monotonicity margin~$m$ emerges as the single quantity governing robustness to quantization: convergence of the forward and backward solvers is guaranteed provided $\norm{\DW}_2 < m$ (Theorem~\ref{thm:weight_quant_bound}), the equilibrium displacement satisfies $\norm{\zqstar - \zstar}_2 \le (\norm{\DW}_2/m)\norm{\zqstar}_2$ (Theorem~\ref{thm:equilibrium_displacement}), and the relative condition number $\kappaRel = \norm{W}_2/m$ links bit-width to forward error (Theorem~\ref{thm:condition_number}). Experiments confirm a phase transition at the predicted threshold and show the displacement bound holds on every tested sample across 6--16 bits, with a conservative factor of $3$--$10\times$. Quantization-aware training recovers convergence at 4~bits where post-training quantization fails, enabled by the backward-pass guarantee of Theorem~\ref{thm:backward_quantized}.

The analysis is limited to uniform symmetric quantization of a single-layer MonDEQ; natural extensions include per-channel and mixed-precision schemes, multi-layer architectures, and margin-aware regularization. An important open question is whether the structural guarantees of equilibrium-based control components survive weight quantization. Recurrent equilibrium networks (RENs)~\cite{revay_recurrent_2024} --- the related dynamic architecture in which equilibrium-based controllers are currently deployed~\cite{junnarkar_synthesis_2022, wang_learning_2022} --- are the natural next target, and the bounds here are a first step. Another avenue for future work is to extend these results to realizations of MonDEQs on quantized analog hardware \cite{chaffey_circuit_2025}.


\section*{ACKNOWLEDGMENT}

Generative AI was used to assist with the experimentation code, finding references, and checking for grammatical errors.

\bibliographystyle{IEEEtran}
\bibliography{bibliography}

\end{document}